\documentclass[a4paper]{jpconf}
\usepackage{graphicx}
\usepackage{amsmath,amsthm}
\usepackage{amssymb}
\usepackage{bm}
\usepackage{esint} 
\usepackage{hyperref}
\usepackage{subcaption}
\usepackage{anyfontsize}

\newtheorem{proposition}{Proposition}
\newtheorem{lemma}{Lemma}
\newtheorem{remark}{Remark}
\newtheorem{theorem}{Theorem}

\usepackage{algorithmicx}
\usepackage[ruled,vlined]{algorithm2e}
\usepackage{algpseudocode,float}
\algblock{ParFor}{EndParFor}

\usepackage{color}

\newcommand{\R}{\mathbb{R}}
\def\red #1{{\color{red}#1}}

\let\red\relax

\begin{document}
\title{\red{Alternating Direction Implicit (ADI) schemes for a PDE-based image osmosis model}}

\author{L. Calatroni$^1$, C. Estatico$^2$, N. Garibaldi$^2$, S. Parisotto$^3$}

%\address{$^1$CMAP, Ecole Polytechnique 91128 Palaiseau Cedex, FRANCE}
%\address{$^2$ Dipartimento di Matematica, Universit\`a degli studi di Genova, Via Dodecaneso 35, 16146, Genova, ITALY}
%\address{$^3$ Cambridge Centre for Analysis, Wilberforce Road, CB3 0WA, University of Cambridge, UK}

\address{$^1$ CMAP, Ecole Polytechnique 91128 Palaiseau Cedex, France}
\address{$^2$ Dipartimento di Matematica, Universit\`a di Genova, Via Dodecaneso 35, 16146, Genova, Italy}
\address{$^3$ Cambridge Centre for Analysis, Wilberforce Road, CB3 0WA, University of Cambridge, UK}

\ead{luca.calatroni@polytechnique.edu, estatico@dima.unige.it, garibaldi91nicola@gmail.com, sp751@cam.ac.uk}

\begin{abstract}
We consider \emph{Alternating Direction Implicit} (ADI) splitting schemes to compute efficiently the numerical solution of the PDE osmosis model considered by Weickert et al.\ in \cite{weickert} for several imaging applications. The discretised scheme is shown to preserve analogous properties to the continuous model. The dimensional splitting strategy traduces numerically into the solution of simple tridiagonal systems for which standard matrix factorisation techniques can be used to improve upon the performance of classical implicit methods, even for large time steps. Applications to the shadow removal problem are presented.
\end{abstract}

\section{Introduction}
Several imaging tasks can be formulated as the problem of finding a reconstructed version of a degraded and possibly under-sampled image. Removing signal oscillations (\emph{image denoising}) and optical aberrations (\emph{image deblurring}) due to acquisition and transmission faults are common problems in applications such as medical imaging, astronomy, microscopy and many more. Further, the  image interpolation problem of filling in missing  or occluded parts of the image using the information from the surrounding areas is called \emph{image inpainting} \cite{aubert,schoenlieb}. \red{More sophisticated imaging tasks such as shadow removal aim to decompose the acquired image into its component structures, such as cartoon and texture, or into the product of an ``intrinsic" component not subject to light-changes (e.g. illumination) and its counterpart describing illumination changes only, see \cite{finlayson,weiss}. 
%This latter task turns out to be helpful for the application of image gradient and thresholding-based techniques for image segmentation.
%since the use of an illumination-homogeneous image representing intrinsically (i.e. independently on light changes) the object of interest makes the segmentation easier in practice \cite{weiss}. 
For a given image $f$, the problem is}
\begin{equation}  \label{eq:inv_prob}
\text{ find } u\quad \text{s.t}.\quad f(x,y)=l(x,y)u(x,y),\qquad (x,y)\in\Omega,
\end{equation}
where $\Omega\subset\R^2$ denotes the image domain, $l(x,y)$ the illumination image and $u$ the \emph{reflectance} image, i.e.\ the intrinsic image not subject to illumination changes \cite{weiss}.
% where $K$ is the (linear) operator modelling the degradation process $u$ goes through and $n$ stands for an additional noise component. Examples for $K$ are classical convolution operators for image deblurring, sampling operators for image inpainting and multiplicative operators modelling light variations for image decomposition and shadow removal problems described above. Due to the unboundedness of $K^{-1}$, the inverse problem in \eqref{eq:inv_prob} is generally ill-posed, so it requires regularisation.
Due to the ill-posedness of the problem \eqref{eq:inv_prob}, in \cite{weiss} a maximum-likelihood approach is used. In \cite{vogel,weickert}, a PDE-based approach computing the solution of \eqref{eq:inv_prob} as the stationary state of a linear parabolic PDE is considered.  Due to its similarities with the physical process, such model is called \emph{image osmosis}.
% and it is briefly recalled in the following.

\paragraph{Image osmosis.} Given a rectangular image domain $\Omega \subset\R^2$ with Lipschitz boundary $\partial \Omega$ and a positive initial gray-scale image $f:\Omega\to\R^+$, for a given vector field $\bm{d}:\Omega\to\R^2$ the linear image osmosis model proposed in \cite{vogel,weickert} is a drift-diffusion PDE which computes for every $t\in (0,T],~ T>0$ a regularised family of images $\left\{u(x,t)\right\}_{t>0}$ of $f$ by solving:
\begin{equation} \label{eq:osmosis}
\begin{cases}
\partial_t u = \Delta u - \text{div}(\bm{d}u) & \text{ on } \Omega \times (0,T] \\
u(x,0)=f(x) & \text{ on } \Omega \\
\langle \nabla u - \bm{d}u, \bm{n} \rangle = 0 & \text{ on } \partial \Omega \times (0, T],
\end{cases}
\end{equation}
where $\langle \cdot, \cdot \rangle$ denotes the Euclidean scalar product and $\bm{n}$ the outer normal vector on $\partial\Omega$. The extension to RGB images is straightforward by letting the process evolve for each colour channel.  In \cite[Proposition 1]{weickert} the authors prove that any solution of \eqref{eq:osmosis} preserves the average gray value and the non-negativity at any time $t>0$. Furthermore, by setting $\bm{d}:=\bm{\nabla} (\ln  v)$ for a reference image $v>0$, the steady state $w$ of $\eqref{eq:osmosis}$ turns out to be a minimiser of \red{a quadratic energy functional} 
%$E(u) = \int_\Omega v \Big| \bm{\nabla} \Bigl(\frac{u}{v}\Bigr) \Big|^2~dx$,  
and can be expressed as a rescaled version of $v$ via the formula $w(x)=\frac{\mu_f}{\mu_v}~v(x)$, where $\mu_f$ and $\mu_v$ are the average of $f$ and $v$ over $\Omega$, respectively. 

%\begin{theorem}{\cite{weickert}}  \label{theor:weickert_cont}
%Any solution of \eqref{eq:osmosis} satisfies the following properties:
%\begin{itemize}
%\item \emph{\textbf{Average gray value conservation}}: $\fint_\Omega u(x,t)~dx =\fint_\Omega f(x)~dx$ for every $t>0$.
%\item \emph{\textbf{Non-negativity}}: $u(x,t)\geq 0$ for every $x\in\Omega$ and $t>0$.
%\item \emph{\textbf{Variational formulation and steady state}}: given a positive image $v>0$ and $\bm{d}=\ln(\nabla v)$ the steady state $w(x)$ of \eqref{eq:osmosis} is a minimiser of 
%$$
%E(u) = \int_\Omega v \Big| \nabla \Bigl(\frac{u}{v}\Bigr) \Big|^2~dx.
%$$
%Moreover, $w(x)=\frac{\mu_f}{\mu_v}~v(x)$, where $\mu_f:=\fint_\Omega f(x)~dx$ and $\mu_v:=\fint_\Omega v(x)~dx$.
%\end{itemize}
%\end{theorem}

In \cite{weickert} the osmosis model \eqref{eq:osmosis} is used as a non-symmetric variant of diffusion for several visual computing applications (data compression, face cloning and shadow removal). A general fully discrete framework for this problem is studied in \cite{vogel} where, under some standard conditions on the semi-discretised finite differences operators used, the conservation properties described above are shown to be still valid upon discretisation. We synthesise these results as follows:

\begin{theorem}[\cite{vogel}]\label{theo:vogel}
For a given  $\bm{f}\in\R^N_+$, consider the semi-discretised linear osmosis problem:
\begin{equation}   \label{eq:semi_discrete}
\bm{u}(0) = \bm{f}, \qquad \bm{u}'(t) =\bm{A} \bm{u}(t),\qquad t>0,
\end{equation}
where $\bm{A}\in \R^{N\times N}$ is an irreducible (non-symmetric) matrix with zero-column sum and non-negative off-diagonal entries. For $k\geq 0$, consider the time-discretisation schemes $\bm{u}^{k+1}=\bm{P}\bm{u}^k$:
%\begin{itemize}
%\item[-] \emph{\textbf{Forward Euler}}: $\bm{P}:=(\bm{I}+\tau \bm{A})$, for $\tau < \frac{1}{\max{|a_{i,i}|}}$;
%\item[-] \emph{\textbf{Backward Euler}}: $\bm{P}:=(\bm{I}-\tau \bm{A})^{-1}$, for any $\tau>0$.
%\end{itemize}
\vspace{-0.6cm}
\begin{align*}
\bm{P}&:=(\bm{I}+\tau \bm{A}),&&\text{for } \tau < (\max{|a_{i,i}|})^{-1}; \tag*{\textbf{\emph{Forward Euler} (F.E.)}}\\
\bm{P}&:=(\bm{I}-\tau \bm{A})^{-1},&&\text{for any }\tau>0.\tag*{\textbf{\emph{Backward Euler} (B.E.)}}
\end{align*}
\vspace{-0.6cm}

Then, in both cases, $\bm{P}$ is an irreducible, non-negative matrix with strictly positive diagonal entries and unitary column sum and for $\bm{u}(t)$ the following properties hold true:
\begin{enumerate}
\item For every $t>0$, the evolution preserves positivity and the average gray value of $\bm{f}$;
\item The unique steady state is the eigenvector of $\bm{P}$ associated to eigenvalue $1$.
\end{enumerate}
\end{theorem} 

For large $\tau>0$, the numerical realisation via \red{B.E.}\ requires the inversion of a non-symmetric penta-diagonal matrix, which may be highly costly for large images. \red{In this paper we solve problem \eqref{eq:semi_discrete} using accurate dimensional semi-implicit splitting methods, \cite{hundsbook}. Similarly, fully implicit splitting schemes are considered in \cite{parisotto_art} and applied to large images for cultural heritage.}

\paragraph{ADI splitting methods.} 
%Let us consider the semi-discretised initial boundary value problem \eqref{eq:semi_discrete} on a domain $\Omega\in\R^s$.
%We follow \cite{hundsbook} and consider a  large system of ODEs arising from a semi-discretisation of initial boundary value problems 
%\[
%\left\{\begin{aligned}
%U'(t)&=F(t,U(t))\quad\text{for }t\geq 0, \\
%U(0)&=U_0,
%\end{aligned}\right.
%\]
%for a given function $F: (0,T)\times \Omega\to\R^s$, $U:(0,T)\to\R^s, \Omega\subset\R^s$ and initial condition $U_0\in\R^s$. The explicit dependence on time of $F$ may not appear, thus considering \emph{autonomous} systems. 
%Let us consider the semi-discretised initial boundary value problem \eqref{eq:semi_discrete} on a domain $\Omega\in%\R^s$. A \emph{directional splitting method} decomposes the differential operator $\bm{A}$ into the sum:
Given a domain $\Omega\in\R^s$, a \emph{dimensional splitting} method decomposes the semi-discretised operator $\bm{A}$ of initial boundary value problem \eqref{eq:semi_discrete} into the sum:
\begin{equation} \label{decomp}
\bm{A}=\bm{A_0}+\bm{A_1}+\ldots+\bm{A_s}
\end{equation}
such that the components $\bm{A_j}$, $j=1,\ldots, s,$ encode the linear action of $\bm{A}$ along the space direction $j=1,\ldots, s,$ respectively. The term $\bm{A_0}$ may contain additional contributes coming from mixed directions and non-stiff nonlinear terms. \emph{Alternating directional implicit} (ADI) schemes are time-stepping methods that treat the unidirectional components $\bm{A_j}, j\geq 1$ implicitly and the $\bm{A_0}$ component, if present, explicitly in time. Having in mind a standard finite difference space discretisation, this splitting idea translates into reducing the $s$-dimensional original problem to $s$ one-dimensional problems. 

For imaging applications $s=2$. Let $u_{i,j}$ denote an approximation of $u$ in the grid of size $h$ at point $((i-\frac{1}{2})h, (j-\frac{1}{2})h)$; the discretisation of \eqref{eq:osmosis} considered in \cite{vogel,weickert}  reads:
\begin{align}  \label{eq:discr_operators}
%&  u_{i+1,j} \bigg( \frac{1}{h^2}-\frac{d_{1,i+\frac{1}{2},j}}{2h} \bigg)+ u_{i-1,j} \bigg( \frac{1}{h^2}+\frac{d_{1,i-\frac{1}{2},j}}{2h} \bigg) + 
%u_{i,j} \bigg( -\frac{2}{h^2}+\frac{d_{1,i-\frac{1}{2},j}-d_{1,i+\frac{1}{2}}}{h} +\frac{d_{2,i,j-\frac{1}{2}}-d_{2,i,j+\frac{1}{2}}}{h}\bigg)\\ 
% + & u_{i,j+1} \bigg( \dfrac{1}{h^2}-\dfrac{d_{2,i,j+\frac{1}{2}}}{2h} \bigg)+
%u_{i,j-1} \bigg( \dfrac{1}{h^2}+\dfrac{d_{2,i,j-\frac{1}{2}}}{2h} \bigg) + u_{i,j} \bigg( -\dfrac{2}{h^2}-\dfrac{d_{1,i+\frac{1}{2},j}}{h} +\dfrac{d_{1,i-\frac{1}{2},j}}{h}-\dfrac{d_{2,i,j+\frac{1}{2}}}{h}+\dfrac{d_{2,i,j-\frac{1}{2}}}{h} \bigg).
%
% 
u'_{i,j} & = \dfrac{u_{i+1,j}-2u_{i,j}+u_{i-1,j}}{h^2} -\dfrac{1}{h} \bigg( d_{1,i+\frac{1}{2},j}\dfrac{u_{i+1,j}+u_{i,j}}{2}-d_{1,i-\frac{1}{2},j}\dfrac{u_{i,j}+u_{i-1,j}}{2} \bigg) \\ 
& + \dfrac{u_{i,j+1}-2u_{i,j}+u_{i,j-1}}{h^2} -\dfrac{1}{h} \bigg( d_{2,i,j+\frac{1}{2}}\dfrac{u_{i,j+1}+u_{i,j}}{2}-d_{2,i,j-\frac{1}{2}}\dfrac{u_{i,j}+u_{i,j-1}}{2} \bigg) =: (\bm{A}_1 + \bm{A}_2) \bm{u}(t)\notag
\end{align}
%We clearly observe that such discretisation is easily decomposable into the sum of two one-directional operators $\bm{A_1}$ and $\bm{A_2}$ with 
which can be splitted into the sum of $\bm{A_1}$ and $\bm{A_2}$ where no mixed and/or nonlinear terms appear ($\bm{A_0}=\bm{0}$). Following \cite{hundsbook}, we recall now the two ADI schemes considered in this paper. 

\paragraph{The Peaceman-Rachford scheme.} The first method considered is the second-order accurate \emph{Peaceman-Rachford} ADI scheme.
%where the operator $\bm{A}$ is splitted into $\bm{A}=\bm{A_1}+\bm{A_2}$,
For every $k\geq 0$ and time-step $\tau>0$, compute an approximation $\bm{u}^{k+1}$ via the updating rule:
\begin{equation}  \label{eq: peacemanrach}
\left\{\begin{aligned}
\bm{u}^{k+1/2}&=\bm{u}^{k}+\frac{\tau}{2}\bm{A_1u}^{k}+\frac{\tau}{2}\bm{A_2u}^{k+1/2}, \\
 \bm{u}^{k+1}&=\bm{u}^{k+1/2}+\frac{\tau}{2}\bm{A_1u}^{k+1}+\frac{\tau}{2}\bm{A_2u}^{k+1/2},
\end{aligned}\right.
\end{equation}
where \red{F.E.}\ and \red{B.E.}\ are applied alternatively and in a symmetric way.

\paragraph{The Douglas scheme.} A more general ADI decomposition accommodating also the general case $\bm{A_0}\neq \bm{0}$ in \eqref{decomp} is the \emph{Douglas} method. For $k\geq 0$, $\tau>0$ and $\theta\in [0,1]$, the updating rule reads
\begin{equation} \label{eq: Douglas}
\left\{ \begin{aligned}
\bm{y}^0&=\bm{u}^k+\tau \bm{Au}^k & \\
\bm{y}^j &=\bm{y}^{j-1}+\theta\tau(\bm{A_j y}^j-\bm{A_j u}^k),\quad &j=1,2 \\
\bm{u}^{k+1}&=\bm{y}^2.&
\end{aligned}\right.
\end{equation}
In words, the numerical approximation in each time step is computed by applying at first a \red{F.E.} predictor and then it is stabilised by intermediate steps where just the unidirectional components $\bm{A_j}$ of the splitting \eqref{decomp} appear, weighted by $\theta$ whose size balances the implicit/explicit behaviour of these steps. %In other words, the unidirectional operators are applied to the convex combination $\theta  U_{n+1}+(1-\theta) U_n$, thus considering fully implicit steps for $\theta=1$, explicit ones for $\theta=0$ and a Crank-Nicolson type scheme for $\theta=1/2$.
The time-consistency order of the scheme is equal to two for $\bm{A_0}=\bm{0}$ and $\theta=1/2$, and it is of order one otherwise. 

\smallskip

For $s=2$ both schemes \eqref{eq: peacemanrach} and \eqref{eq: Douglas} are unconditionally stable, see \cite{hundsbook}. However, for large time steps, the time-accuracy may suffer due to the presence of the explicit steps.
% As remarked \cite[Page 370-373]{hundsbook}, due the consistency of the internal stages, both Peaceman-Rachford and Douglas methods return stationary solutions exactly.

%\textcolor{red}{
%\begin{remark}
%For $\theta =1$, we can avoid the computation of $Y_0$ since $Y_1$ is equivalent to 
%\begin{align*}
%Y_1 &= U_n + \Delta t F (t_n,U_n) + \Delta t F_1(t_{n+1},Y_1) - \Delta t F_1(t_{n},U_n)\\
%&= U_n + \Delta t F_1 (t_n,U_n)  + \Delta t F_2 (t_n,U_n)  + \Delta t F_1(t_{n+1},Y_1) - \Delta t F_1(t_{n},U_n)\\
%&= U_n  + \Delta t F_2 (t_n,U_n)  + \tau F_1(t_{n+1},Y_1) 
%\end{align*}
%i.e.\
%\[
%Y_1  - \tau F_1(t_{n+1},Y_1)  = U_n  + \tau F_2 (t_n,U_n).
%\]
%\end{remark}
%}

\paragraph{Scope of the paper.} In this work, we apply schemes \eqref{eq: peacemanrach} and \eqref{eq: Douglas} to solve \eqref{eq:semi_discrete} and show that similar conservation properties as in Theorem \ref{theo:vogel} hold. Compared to standard implicit methods used in \cite{vogel,weickert}, such schemes renders a much more efficient and accurate numerical solution.

\section{ADI methods for image osmosis}  \label{sec:ADI_theor}

We show that ADI schemes \eqref{eq: peacemanrach} and \eqref{eq: Douglas} show similar conservation properties as in Theorem \ref{theo:vogel}.
 
\begin{proposition} \label{prop:pr}
Let  $\bm{f}\in\R^N_+$ and $\tau < 2\Bigl(\max\left\{ \max |a^1_{i,i}| , \max |a^2_{i,i}| \right\}\Bigr)^{-1}$. Then, the Peaceman-Rachford scheme \eqref{eq: peacemanrach} on the splitting \eqref{eq:discr_operators} preserves the average gray value, the positivity and converges to a unique steady state.
\end{proposition}

\begin{proof}
We write the Peaceman-Rachford \eqref{eq: peacemanrach} iteration as
\[
\bm{u}^{k+1} = \left(\bm{I} -\frac{\tau}{2} \bm{A_1}\right)^{-1} \left(\bm{I} +\frac{\tau}{2} \bm{A_2}\right) \left(\bm{I} -\frac{\tau}{2} \bm{\red{A_2}}\right)^{-1} \left(\bm{I} +\frac{\tau}{2}\bm{A_1}\right) \bm{u}^{k} ,
\]
and observe that for $i=1,2$ the matrices $\bm{P^-_i}:=\left(\bm{I} -\frac{\tau}{2} \bm{A_i}\right)^{-1}$ and $\bm{P^+_i}:=\left(\bm{I} +\frac{\tau}{2} \bm{A_i}\right)$, are non-negative and irreducible with unitary column sum and positive diagonal entries being each $\bm{A_i}$ a one-dimensional implicit/explicit discretised osmosis operator satisfying the assumptions of Theorem \ref{theo:vogel}. Therefore, at every implicit/explicit half-step and using the restriction on $\tau$, the average gray-value and positivity are conserved. Furthermore, the unique steady state is the eigenvector of the operator $\bm{P}:=\bm{P^-_1}\bm{P^+_2}\bm{P^-_2}\bm{P^+_1}$ associated to the eigenvalue to one.
\qedhere
\end{proof}

%In order to prove a similar result for the Douglas scheme \eqref{eq: Douglas}, we  first prove the following auxiliary result.
The following lemma is useful to prove a similar result for the Douglas scheme \eqref{eq: Douglas}.

%\begin{lemma}   \label{lemma:sum}
%Let $\bm{C}=(c_{i,j})\in\R^{N\times N}$ and $\bm{D}=(d_{i,j})\in\R^{N\times N}$ such that:
%\[
%\sum_{i=1}^N c_{i,j} =: c,\qquad \sum_{i=1}^N d_{i,j} =: d, \qquad \text{for every }j=1,\ldots,N.
%\]
%Then, the matrix $\bm{B}:=\bm{CD}$ has column sum equal to $cd$.
%% i.e. for every column $j=1,\ldots,N$
%%$$
%%\sum_{i=1}^N b_{i,j} = cd.
%%$$
%\end{lemma}
\begin{lemma}   \label{lemma:sum}
If $\bm{C}=(c_{i,j})\in\R^{N\times N}$ and $\bm{D}=(d_{i,j})\in\R^{N\times N}$ s.t. $
\sum_{i=1}^N c_{i,j} =: c$ and $\sum_{i=1}^N d_{i,j} =: d$ for every $j=1,\ldots,N$, then, the matrix $\bm{B}:=\bm{CD}$ has column sum equal to $cd$.
% i.e. for every column $j=1,\ldots,N$
%$$
%\sum_{i=1}^N b_{i,j} = cd.
%$$
\end{lemma}
\begin{proof}
Writing each element  $b_{i,j}$ in terms of the elements of $\bm{C}$ and $\bm{D}$, we have that for every $j$:
\[
\sum_{i=1}^N b_{i,j} = \sum_{i=1}^N\sum_{k=1}^N c_{i,k}d_{k,j} = \sum_{k=1}^N\Bigl( d_{k,j} \sum_{i=1}^N c_{i,k}\Bigr) = \sum_{k=1}^N \Bigl( d_{k,j}\cdot c \Bigr) = c \sum_{k=1}^N d_{k,j} = c d.\qedhere
\]
\end{proof}

\begin{proposition}  \label{propop:agv}
Let $f\in\R^{N}_+$, $\tau>0$ and $\theta\in[0,1]$. Then, the Douglas scheme \eqref{eq: Douglas} applied to the splitted semi-discretised scheme \eqref{eq:discr_operators} preserves the average gray value.
\end{proposition}
\begin{proof}
For every $k\geq 0$, we write the Douglas iteration \eqref{eq: Douglas} as
\begin{align*}
\bm{u}^{k+1} &  = \left( \bm{I}-\theta \tau \bm{A_2} \right)^{-1}  \left[ (\bm{I}-\theta \tau \bm{A_1})^{-1} \Bigl( (\bm{I}-\theta \tau \bm{A_1})+\tau \bm{A}\Bigr)-\theta \tau \bm{A_2}\right] \bm{u}^k  \\
& = \Bigl(\bm{I} + \tau (\bm{I}-\theta \tau \bm{A_2})^{-1} (\bm{I}-\theta \tau \bm{A_1})^{-1}\bm{A} \Bigr) \bm{u}^k = :\Bigl(\bm{I} + \tau\bm{P_2}\bm{P_1}\bm{A} \Bigr) \bm{u}^k
\end{align*}
where both $\bm{P_1}$ and $\bm{P_2}$ are non-negative, irreducible, with unitary column sum and strictly positive diagonal entries by Theorem \ref{theo:vogel}, while $\bm{A}=\bm{A_1}+\bm{A_2}$ is zero-column sum being the standard discretised osmosis operator \eqref{eq:discr_operators}. By Lemma \ref{lemma:sum}, the operator $\bm{B}:=\bm{P_2P_1A}$ has zero column sum and the operator $\bm{P}:=\bm{I}+\tau\bm{B}$ has unitary column sum. Thus, for every $k\geq 0$:
\[
\frac{1}{N} \sum_{i=1}^N u^{k+1}_i = \frac{1}{N} \sum_{i=1}^N \sum_{j=1}^N p_{i,j}u^{k}_j = \frac{1}{N} \sum_{j=1}^N  \Bigl(\sum_{i=1}^N  p_{i,j}\Bigr)u^{k}_j =  \frac{1}{N} \sum_{j=1}^N u_j^k.\qedhere
\]
\end{proof}

\begin{remark}
There is no guarantee that the off-diagonal entries of $\bm{B}$ are non-negative. Therefore, it is not possible to apply directly Theorem \ref{theo:vogel} to conclude that the iterates $\left\{\bm{u}^k\right\}_{k\geq 0}$ remain positive and converge to a unique steady state. However, our numerical tests suggest that both properties remain valid. A rigorous proof of these properties is left for future research. 
\end{remark}

\section{Numerical results} \label{sec:ADI_num}
We present in Algorithm \ref{alg: peaceman LU} and \ref{alg: douglas LU} the pseudocode for the Peaceman-Rachford \eqref{eq: peacemanrach} and the Douglas \eqref{eq: Douglas} schemes, respectively, when applied to images of size $m\times n$ pixels, with $N:=mn$.
We recall that in \cite{vogel, weickert}, the non-split problem is solved iteratively by \red{B.E.} using \texttt{BiCGStab}, whose performance  is influenced by the accuracy and maximum iterations required.
 
Due to the structure of the ADI matrices $\bm{A_1}$ and $\bm{A_2}$, a tridiagonal \texttt{LU} factorisation can be used for improved efficiency after a  permutation on matrix $\bm{A_2}$ based on \cite{saad} to reduce the bandwidth to one: this is convenient for images with large $m$, where standard methods become prohibitively expensive. Without splitting, one \texttt{LU} factorisation plus $T\tau^{-1}$ system resolutions for the penta-diagonal matrix $\bm{A}$, with lower and upper bandwidth equal to $m$, costs $\mathcal{O}(2m^3 n + 4T\tau^{-1} m^2n)$ \emph{flops} overall. Using splitting, the computational cost is reduced to $\mathcal{O}(6(m n-1) + T\tau^{-1}(10mn-8))$ \emph{flops} because of the use of tridiagonal $1$-bandwidth matrices. For the same reason, the inverse operators $\bm{A_1}$ and $\bm{A_2}$ require less storage space than the full operator $\bm{A}$ while the inverse of $\bm{A}$ has significantly more non-zeros entries.  Note that since $\tau$ is constant, all the \texttt{LU} factorisations can be computed only once before the main loop. 

\vspace{-0.5em}
\noindent
\begin{minipage}[t]{0.495\textwidth}\null\small
\begin{algorithm}[H]
\caption{LU Peaceman-Rachford}
\label{alg: peaceman LU}
\SetKwData{maxiter}{maxiter}
\SetKwData{colorchannel}{color channel}
\SetKwInOut{Input}{Input}
\SetKwInOut{Parameters}{Parameters}
\SetKwInOut{Identification}{Identification}
\SetKwInOut{Initialization}{Initialization}
\SetKwInOut{Update}{Update}
\Input{$u^0$, $\mathtt{A=A_1+A_2}$;}
\Parameters{$\tau$, $T$;}
%\begin{algorithmic}
%\quad \ForEach{$c$ in \colorchannel \vspace{-0.2cm}}
\ForEach{$c$ in \colorchannel }
{
\vspace{-1.4em}
\begin{align*}
&\mathtt{p}_c = \mathtt{symrcm}(\mathtt{A_2}_c);\\
&\mathtt{E1}_c = \left(\,\mathtt{I} + 0.5\,\tau\,\mathtt{A_1}_c\,\right);\\
&\mathtt{E2}_c  = \left(\,\mathtt{I} + 0.5\,\tau\,\mathtt{A_2}_c(\mathtt{p_c}\,;\,\mathtt{p_c})\,\right);\\
&[\mathtt{L1}_c,\,\mathtt{U1}_c] = \mathtt{lu}(\,\mathtt{I} -  0.5\,\tau\,\mathtt{A_1}_c\,);\\
&[ \mathtt{L2}_c,\, \mathtt{U2}_c] = \mathtt{lu}(\,\mathtt{I} -  0.5\,\tau\,\mathtt{A_2}_c(\mathtt{p_c}\,;\,\mathtt{p_c})\,).
\end{align*}
%\quad \For{$k=0$ \KwTo $T-\tau$ }
\For{$k=0$ \KwTo $T-\tau$ }
{
\vspace{-0.9em}
\begin{align*}
&z_1 = \mathtt{E1}_c \, u_c^k;\\
&y_1(\mathtt{p}_c) = \mathtt{U2}_c\,\backslash\,( \mathtt{L2}_c \,\backslash \,  z_1(\mathtt{p}_c));\\
&z_2(\mathtt{p}_c) = \mathtt{E2}_c \, y_1(\mathtt{p}_c);\\
&u_c^{k+\tau} = \mathtt{U1}_c \,\backslash\, ( \mathtt{L1}_c\,\backslash\,   z_2);
\end{align*}
%\vspace{+0.2em}
}
}
%\end{algorithmic}
\Return{$u^{T}$.}
\end{algorithm}
\end{minipage}
\begin{minipage}[t]{0.495\textwidth}\null\small
\begin{algorithm}[H]
%\begin{algorithm}[!htb]
\caption{LU-Douglas}
\label{alg: douglas LU}
\SetKwData{maxiter}{maxiter}
\SetKwData{colorchannel}{color channel}
\SetKwInOut{Input}{Input}
\SetKwInOut{Parameters}{Parameters}
\SetKwInOut{Identification}{Identification}
\SetKwInOut{Initialization}{Initialization}
\SetKwInOut{Update}{Update}
\Input{$u^0$, $\mathtt{A=A_1+A_2}$;}
\Parameters{$\tau$,  $T$, $\theta\in [0,1]$;}
%\begin{algorithmic}
%\quad \ForEach{$c$ in \colorchannel \vspace{-0.49cm}}
\ForEach{$c$ in \colorchannel}
{
\vspace{-1.4em}
\begin{align*}
&\mathtt{p}_c  = \mathtt{symrcm}(\mathtt{A_2}_c);\\
&\mathtt{E1}_c = \tau\,\mathtt{A_1}_c\,; ~~  \mathtt{E2}_c  = \,\tau\,\mathtt{A_2}_c(\mathtt{p_c}\,;\,\mathtt{p_c})\,\\
&\mathtt{T1}_c = \tau\,\theta\,\mathtt{A_1}_c\,;  \mathtt{T2}_c  = \tau\,\theta\,\mathtt{A_2}_c(\mathtt{p_c}\,;\,\mathtt{p_c})\,;\\
&[\mathtt{L1}_c,\,\mathtt{U1}_c] = \mathtt{lu}(\,\mathtt{I} -  \tau\,\theta\,\mathtt{A_1}_c\,);\\
&[ \mathtt{L2}_c,\, \mathtt{U2}_c] = \mathtt{lu}(\,\mathtt{I} -  \tau\,\theta\,\mathtt{A_2}_c(\mathtt{p_c}\,;\,\mathtt{p_c})\,).
\end{align*}
%\quad \For{$k=0$ \KwTo $T-\tau$ \vspace{-0.5cm}}
\For{$k=0$ \KwTo $T-\tau$}
{
\vspace{-1.4em}
\begin{align*}
& y_{01} = \mathtt{E1}_c \, u_c^k\,; ~~
y_{02}(\mathtt{p}_c) = \mathtt{E2}_c \, u_c^k(\mathtt{p}_c)\,;\\ 
%&y_{0} = u_c^k + y_{01} + y_{02};\\
&z_1 = u_c^k + y_{01} + y_{02}-\mathtt{T1}_c\, u_c^k;\\
&y_1 = \mathtt{U1}_c\,\backslash\,( \mathtt{L1}_c \,\backslash \,  z_1);\\
&z_2(\mathtt{p}_c) = y_1(\mathtt{p}_c) - \mathtt{T2}_c\, u_c^k;\\
&u_c^{k+\tau}(\mathtt{p}_c)  = \mathtt{U2}_c \,\backslash\, ( \mathtt{L2}_c\,\backslash   \,z_2(\mathtt{p}_c) );
\end{align*}
\vspace{-2em}
}
}
\vspace{-0.15em}
%\end{algorithmic}
\Return{$u^{T}$.}
\end{algorithm}
\end{minipage}

As a toy example, we let the osmosis model \eqref{eq:osmosis} evolve towards a known steady state starting from an initial constant image. In Table \ref{tab: numerical results} we compare the numerical solutions at \red{$T=5000$} for different fixed time-steps $\tau$ and values of $\theta$. \red{The benchmark solution $\bar{u}$ is computed using Exponential Integrators \cite{higham}, which for linear equations are exact-in-time solvers}. Similarly as in \cite{vogel}, the approach is solved without splitting by \texttt{BiCGStab} (with parameters \red{\texttt{tol=1e-07} and \texttt{maxiter=3e+05})}. For comparison, we test the direct $\texttt{LU}$ factorisation applied to the full problem. The order of the methods for the \red{relative Root Mean Square (RMS) error (rRMSE) $\texttt{rms}(u_T - \bar{u})/\texttt{rms}(\bar{u}$)} are plotted in~Figure~\ref{fig: order and solution comparison}. \red{For large $\tau$, ADI methods are up to} 10$\times$ faster than methods solving the problem without any splitting.
%\begin{table}\small
%\caption{\red{Comparison between full and ADI splitting $p$-order methods}. Input $240\times 250$ pixels.}\vspace{-1em}
%\begin{center}
%\begin{tabular}{l|c|c|cc|cc|cc}
%\br
%\multicolumn{3}{c}{} &\multicolumn{2}{c}{$\tau = 0.1$} & \multicolumn{2}{c}{$\tau = 1$} & \multicolumn{2}{c}{$\tau = 10$} \\
%\mr
%Method & $\theta$  & $p$		& Time & E.E.\ & Time & E.E.\   & Time & E.E.\  \\
%\mr
%\texttt{BiCGStab} (5 d.)                & 1 & 1 & \textbf{76.70 s.} & \textbf{1.84e-10} & 58.08 s. & 2.00e-09 & 19.44 s. & 1.91e-08\\
%\texttt{LU}  (5 d.)                        & 1 & 1 & \textcolor{blue}{\textbf{528.67 s.}} & \textbf{1.84e-10} & 53.31 s. & \textbf{1.84e-09} & 5.46 s. & \textbf{1.85e-08}\\
%\texttt{Douglas} ADI						& 1 & 1 & 80.04 s &  \textbf{1.84e-10} & \textbf{7.30 s.}  & \textbf{1.84e-09} & \textbf{0.73 s.} & 2.85e-08\\ 
%\mr
%\texttt{BiCGStab}  (5 d.)              & 0.5 &  2 & 75.01 s.   & 5.56e-15 & 23.2 s.  & 5.58e-13 & 20.19 s. & 8.00e-09\\
%\texttt{LU}  (5 d.)                       & 0.5 &  2 &  \textcolor{blue}{\textbf{566.89 s.}} & 5.58e-15 &  52.92 s. & 5.58e-13 & 5.27 s. & 8.00e-09\\
%\texttt{Douglas}  ADI                   & 0.5 &  2 & 80.35 s.   & \textbf{2.00e-15} & 9.76 s. & \textbf{2.00e-13} & 0.86 s.     & \textbf{1.96e-11}\\
%\texttt{P}.-\texttt{R.} ADI            & -     & 2 & \textbf{57.41 s.}    & 1.51e-14 & \textbf{5.81 s.}  & 1.51e-12     & \textbf{0.54 s.}& 1.53e-10\\
%\br
%\end{tabular}
%\end{center}
%\label{tab: numerical results}
%\end{table}
\begin{table}[!h]
\caption{\red{Full vs.\ ADI splitting $p$-order methods at T=5000}. Input: $240\times 250$ pixel image.}
\vspace{-2em}
\begin{center}
\scriptsize
\begin{tabular}{l|c|c|cc|cc|cc|cc}
\br
\multicolumn{3}{c}{} &\multicolumn{2}{c}{$\tau = 0.1$} & \multicolumn{2}{c}{$\tau = 1$} & \multicolumn{2}{c}{$\tau = 10$} & \multicolumn{2}{c}{$\tau = 100$}\\
\mr
Method & $\theta$  & $p$		& Time & rRMSE & Time & rRMSE   & Time & rRMSE & Time & rRMSE  \\
\mr
\texttt{BiCGStab} (5 d.)             & 1 & 1 & 584.90 s. & 3.18e-06 & 104.86 s. & 7.74e-06 & 25.73 s. & 6.82e-05 & 13.91 s. & 6.64e-04\\
\texttt{LU}  (5 d.)                      & 1 & 1 & \textcolor{blue}{\textbf{3165.61 s.}} & 3.60e-07 & 310.01 s. & 3.60e-06 & 33.83 s. & 3.60e-05 & 3.07 s. & 3.60e-04\\
\texttt{Douglas} ADI					& 1 & 1 & 529.59 s  &  3.60e-07 & 57.36 s.  & 3.61e-06 & 6.12 s. & 3.67e-05 & 0.62 s. & 1.47e-02\\ 
\mr
\texttt{BiCGStab}  (5 d.)             & 0.5 &  2 & 600.70 s.   & 1.19e-06 & 97.54 s.  & 2.94e-05 & 33.72 s. & 4.45e-06 & 20.60 s. &  2.05e-02\\
\texttt{LU}  (5 d.)                      & 0.5 &  2 &  \textcolor{blue}{\textbf{3652.58 s.}} & 7.93e-12 &  321.95 s. & 3.11e-10 & 32.47 s. & 3.11e-08 & 3.20 s. & 2.05e-02\\
\texttt{Douglas}  ADI                   & 0.5 &  2 & 523.85 s.   & 3.56e-11 & 53.06 s. & 3.56e-09 & 5.18 s.     & 3.56e-07 & 0.52 s. & 2.32e-02\\
\texttt{P}.-\texttt{R.} ADI            & -     & 2 & 463.68 s.    & 6.32e-11 & 36.35 s.  & 6.32e-09    & 3.56 s. & 6.33e-07 & 0.38 s. & \textcolor{blue}{\textbf{8.10e-02}}\\
\br
\end{tabular}
\end{center}
\label{tab: numerical results}
\end{table}
\vspace{-1em}
%\begin{figure}[!h]
%\begin{subfigure}{0.5\textwidth}
%\centering
%\includegraphics[trim={0 4.55cm 0 4.55cm}, height=4.8cm]{\detokenize{./images/loglog_dt_errenergy_order_all.pdf}}
%\caption{ADI orders for Energy Error.}
%\label{fig: order and solution comparison}
%\end{subfigure}
%\begin{figure}[!h]
%\begin{subfigure}[t]{0.5\textwidth}
%\centering
%\includegraphics[width=1\textwidth]{\detokenize{./images/loglog_dt_rmsrelerror_all_review.pdf}}
%\caption{ADI orders for rRMSE (T=5000).}
%\label{fig: order and solution comparison}
%\end{subfigure}
%\begin{subfigure}[t]{0.5\textwidth}
%\centering
%%\includegraphics[trim={0 0 0 0cm}, clip=true, width=0.36\textwidth]{./images/montage3}
%\includegraphics[width=1\textwidth]{./images/montage3}
%\caption{Shadow removal + inpainting.}
%\label{fig:shadow_removal}
%\end{subfigure}
%\caption{\ref{fig: order and solution comparison}: \red{CHANGE Energy error at $T=1000$ for several $\tau$}. \ref{fig:shadow_removal}:  Shadow removal inputs (top); Douglas ADI solution and result after non-local inpainting \cite{arias} (bottom). Parameters: $\tau =10$; $\theta = 0.5$. Image from: \url{http://aqua.cs.uiuc.edu/site/projects/shadow.html}. \vspace{-0.5cm}}
%\end{figure}

In Figure \ref{fig:shadow_removal} the ADI numerical solution of the shadow-removal problem \cite{weickert} is shown, starting from a shadowed image $f$ and setting $\bm{d}=\bm{\nabla}(\ln f)=0$ on the shadow boundary. In order to correct the diffusion artefacts, an inpainting correction (e.g.\ \cite{arias}) is applied.
\begin{figure}
\centering
\begin{subfigure}[t]{0.48\textwidth}
\includegraphics[width=1\textwidth]{\detokenize{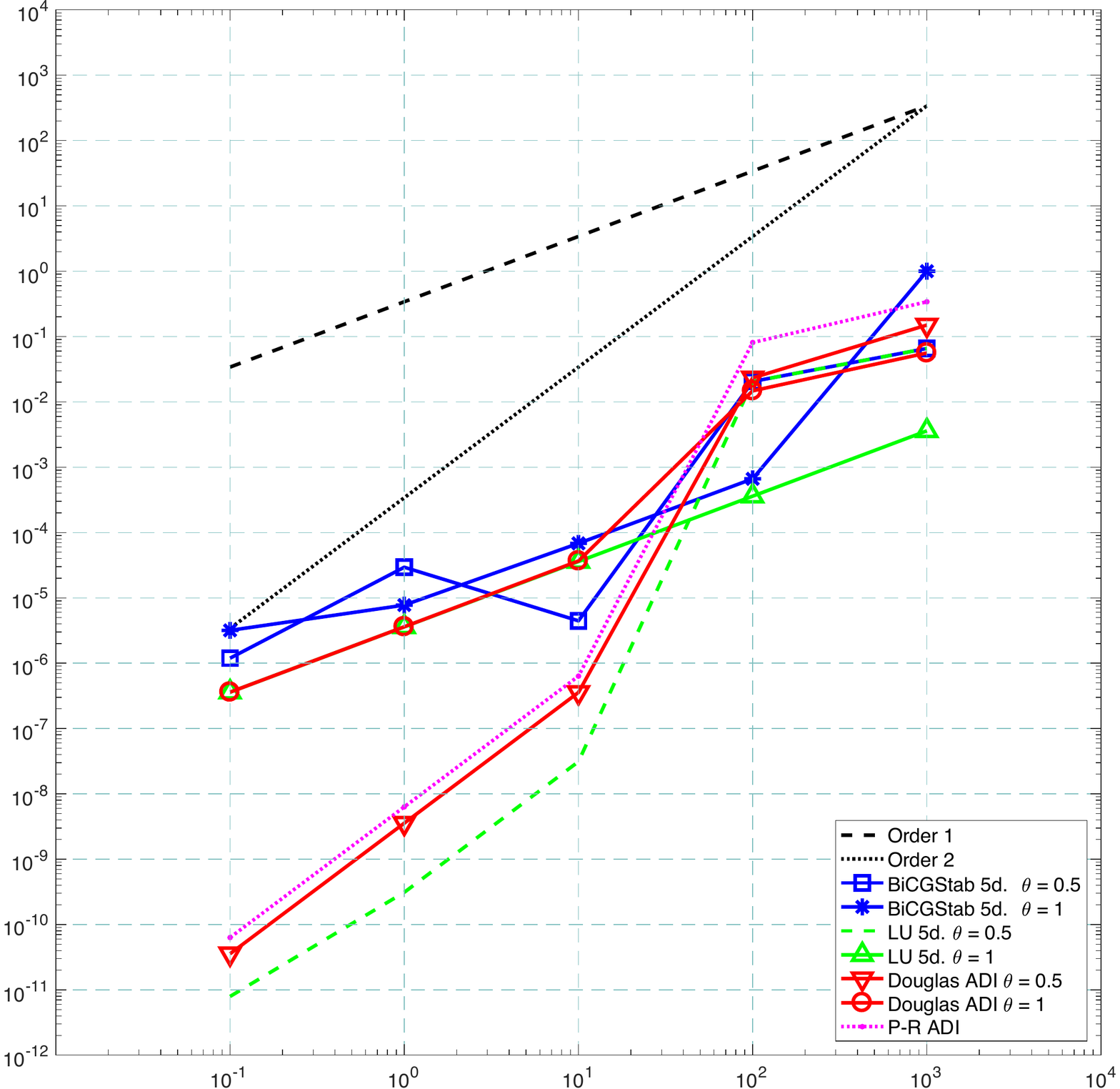}}
\caption{\red{Loglog order plot: $\tau$ vs.\ rRMSE (T=5000)}.\label{fig: order and solution comparison}}
\end{subfigure}\quad
\begin{subfigure}[t]{0.48\textwidth}
\includegraphics[width=1.0\textwidth]{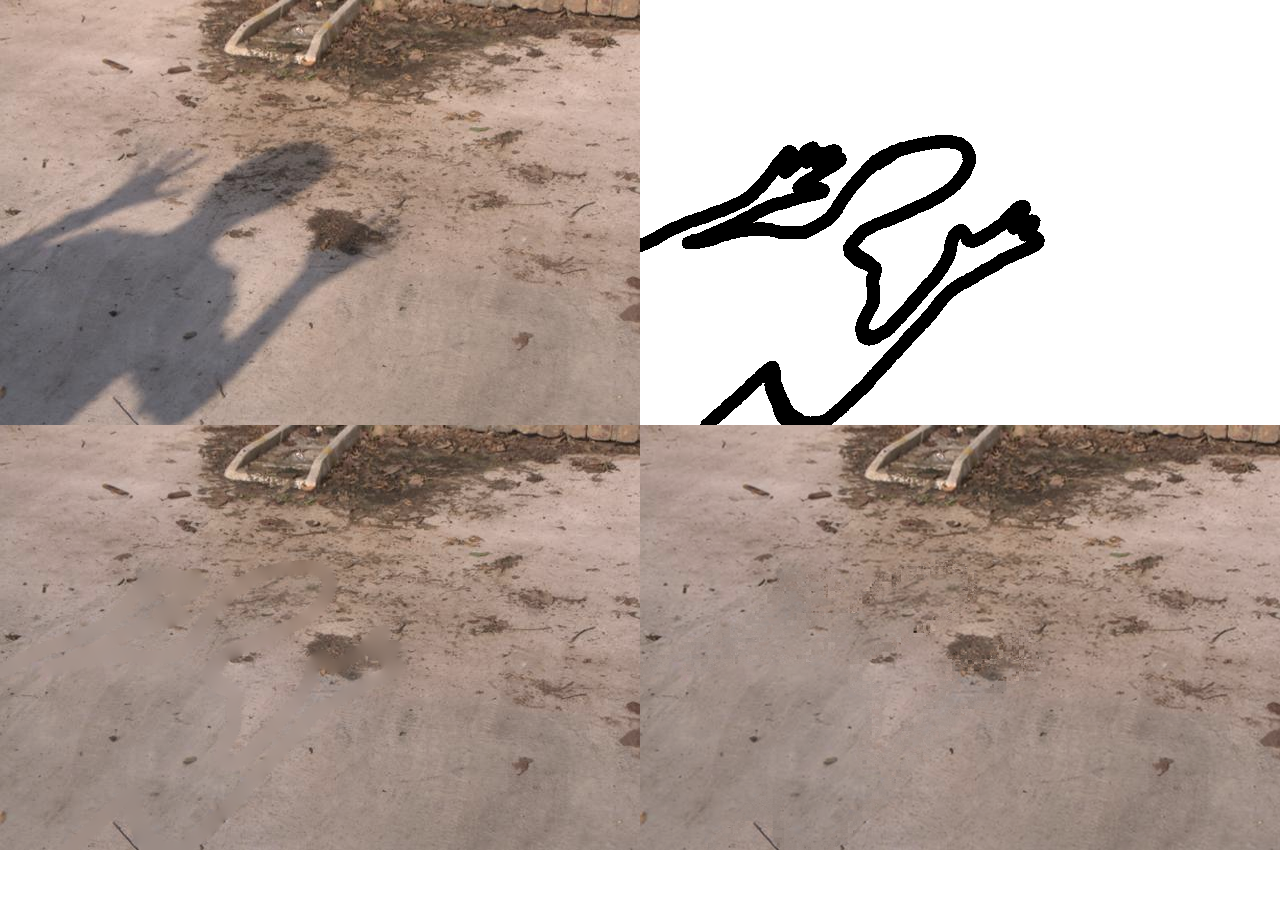}
\caption{\red{Shadow removal and inpainting}.\label{fig:shadow_removal}}
\end{subfigure}
\caption{\ref{fig: order and solution comparison}: \red{Relative RMS error at $T=5000$ for several $\tau$}. \ref{fig:shadow_removal}:  Shadow removal inputs (top); Douglas ADI solution and result after non-local inpainting \cite{arias} (bottom). Parameters: $\tau =10$; $\theta = 0.5$. Image from: \url{http://aqua.cs.uiuc.edu/site/projects/shadow.html}.}
\end{figure}

\section{Conclusions}
In this paper, we consider the \emph{image osmosis} model proposed in \cite{vogel,weickert} for imaging applications. We proposed dimensional splitting approaches to decompose the discretised 2D operator into two tridiagonal operators. We proved that the splitting numerical schemes preserve analogous properties holding in the continuous model. Numerically, these methods are efficient and accurate. We tested the proposed methods for a shadow removal problem, with an inpainting post-processing step. Future research directions include \red{the rigorous convergence proof of the discrete models to the continuum one as $N\to\infty$} \red{as well as} \red{a rigorous stability and convergence analysis for the Douglas method \eqref{eq: Douglas}.}

%the formulation of a (nonlinear) model driving the diffusion on the shadow mask to avoid the inpainting correction. 
\ack
LC acknowledges the joint ANR/FWF Project ``Efficient Algorithms for Nonsmooth Optimization in Imaging" (EANOI) FWF n. I1148 / ANR-12-IS01-0003. CE acknowledges the support of GNCS-INDAM and PRIN 2012 N. 2012MTE38N.
\red{SP acknowledges UK EPSRC grant EP/L016516/1}.

\end{document}